\documentclass[12pt]{article}
\usepackage{graphicx} 
\usepackage{times}
\usepackage{amsmath}
\usepackage{amssymb}
\usepackage{amsthm}
\usepackage{tikz}
\usepackage{subcaption}

\newtheorem{theorem}{Theorem}[section]
\newtheorem{lemma}[theorem]{Lemma}

\newtheorem{proposition}[theorem]{Proposition}

\theoremstyle{definition}

\theoremstyle{remark}

\newcommand{\floor}[1]{\left\lfloor #1 \right\rfloor}
\newcommand{\ceil}[1]{\left\lceil #1 \right\rceil}
\newcommand{\wdim}{\mathrm{wdim}}

\newcommand{\diff}{\bigtriangleup}

\begin{document}

\title{The weak $k$-metric dimension of the direct product of complete graphs}

\author{Mohammad Farhan$^{a,}$\thanks{\texttt{mohammad.farhan@uca.es}}
\and Dorota Kuziak$^{a,}$\thanks{\texttt{dorota.kuziak@uca.es}}
\and Ismael G. Yero$^{b,}$\thanks{\texttt{ismael.gonzalez@uca.es}}
}

\maketitle

\begin{center}
$^a$ Departamento de Estad\'istica e Investigaci\'on Operativa, Universidad de C\'adiz, Algeciras Campus, Spain \\

\medskip
$^b$ Departamento de Matem\'{a}ticas, Universidad de C\'adiz, Algeciras Campus, Spain
\end{center}

\begin{abstract}
The weak $k$-metric dimension of a graph is roughly understood as the cardinality of a smallest set of vertices $S$ of the graph with the property of uniquely recognizing all the vertices of the graph throughout summations of differences of distances to the vertices of $S$. The weak $k$-metric dimension of the direct product of two isomorphic complete graphs is considered in this work. Specifically, the value of such parameter is computed for almost all possibilities of these products and a bound is provided in the remaining case.
\end{abstract}


\noindent
{\bf AMS Subj.\ Class.\ (2020)}:  05C12, 05C76

\section{Introduction}

The weak $k$-metric dimension of graphs was presented in \cite{Peterin} as a variation of the classical metric dimension concept that focuses on making a global influence of a given set to uniquely identify the whole set of vertices of a graph without restricting too much such identification process. This seminal work has immediately attracted the attention, and a continuation of this promising research line, has recently appeared in \cite{Elena}. This latter work deals with computing the weak $k$-metric dimension of the Cartesian product of complete graphs, also known as the $2$-dimensional Hamming graphs, which are classical structures in graph theory. In this sense, a natural continuation on this topic might be that of considering related graph products, like for instance, the diagonal $2$-dimensional Hamming graphs, which appear in connection with the direct product of complete graphs.

The metric dimension of graphs is a classical topic in graph theory whose origin is somehow coming from a related concept in the area of metric spaces (see \cite{Blumenthal}). The specific studies for graphs are understood to be initiated in the two independent works \cite{Slater1975,Harary1976}. However, the popularity of this topic was indeed motivated by the article \cite{Chartrand}, published at the beginning of this century. The investigations on this topic cover a very wide range of lines including classical combinatorial ones, together with other more algorithmic or applied ones. In order to not include so many references to these facts, the reader might simply read the recent surveys \cite{Kuziak,Tillquist}. In addition, some significant recent works on the metric dimension of graphs are \cite{Bailey-2023,Dankelmann-2023,foster-2024,Prabhu}.

Given a connected graph $G=(V(G),E(G))$, a set of vertices $S\subseteq V(G)$ is a \emph{resolving set} for $G$ if for each two vertices $x,y\in V(G)$ there is a vertex $w\in S$ such that $d_G(x,w)$ differs from $d_G(y,w)$, where $d_G(u,v)$ represents the length of a shortest path joining $u$ and $v$ (such notation can be also simply written as $d(x,y)$ if there is no ambiguity). Such length is known as the \textit{distance} between $u$ and $v$. A resolving set of the smallest possible cardinality is called a \emph{metric basis} of $G$, and its cardinality is known as the \emph{metric dimension} of $G$, usually denoted by $\dim(G)$.

As it usually happens, the classical metric dimension concept has been modified in several directions so that more knowledge about it can be discovered (see \cite{Kuziak} that runs along a large number of these variations). One of these variations considered a reinforcement of the identification property of resolving sets, which might have a weakness related to the existence of vertices that are identified by exactly one vertex of the resolving set (see for example \cite{Hakanen} for some contributions in this direction). This reinforcement was somehow solved with the generalized version of the metric dimension called $k$-metric dimension (see \cite{Estrada-Moreno2013}). However, such variation seems to be very restrictive and thus somehow less useful, mainly due to the use of a larger number of ``local'' elements in the identification process.

To solve this issue, the authors of \cite{Peterin} presented an attempt to involve a resolving related set of vertices in which the identification will be made by the whole such resolving set. That is as follows. Consider a set $S\subseteq V(G)$, and three vertices $x,y,z\in V(G)$. Let 
$$\Delta_z(x,y)=|d_G(x,z)-d_G(y,z)|\,,$$
and
$$\Delta_S(x,y)=\sum_{z\in S}\Delta_z(x,y).$$ 
Given an integer $k\ge 1$, the set $S$ is known as a \emph{weak $k$-resolving set} for $G$ if it is satisfied that $\Delta_S(x,y)\ge k$ for each two vertices $x,y\in V(G)$. In the notation $\Delta_S(x,y)$, if $S=V(G)$, then we simply write $\Delta(x,y)$.

In this sense, the \emph{weak} $k$-\emph{metric dimension} of $G$, denoted by $\wdim_{k}(G)$, represents the cardinality of a smallest weak $k$-resolving set of $G$, and such weak $k$-resolving set with the smallest possible cardinality is called a \emph{weak $k$-metric basis} for $G$. These concepts were first described in \cite{Peterin}, as a promising research line that attempts to improve some lacks of potential applicability that are present in the generalized $k$-metric dimension already known from \cite{Estrada-Moreno2013}. 

As it was already noted from the seminal work \cite{Peterin}, a graph $G$ does not have weak $k$-resolving sets for each integer $k\ge 1$. This gives step to represent by $\kappa(G)$ the largest integer $k$ such that $G$ contains a weak $k$-resolving set. Regarding this, it is also said in the literature that a graph $G$ is {\em weak $\kappa(G)$-metric dimensional}. In \cite{Peterin}, it was proved that indeed
\begin{equation}\label{eq-kappa}
    \kappa(G)=\min\left\{\Delta(x,y)\;:\;x,y\in V(G)\right\}.
\end{equation}

The weak $k$-metric dimension of graphs has been studied in the first work \cite{Peterin} for some classical graphs classes including trees and grid graphs (Cartesian products of paths). Further on, in \cite{Elena}, the parameter was also considered for the Cartesian product of two complete graphs (also known as $2$-dimensional Hamming graphs). This latter work presented formulas for the weak $k$-metric dimension of such graphs, together with some improvement of a known integer linear programming formulation for computing this parameter, and its application for the case of $2$-dimensional Hamming graphs.

In this work, we continue this research line, and specifically compute the weak $k$-metric dimension of the direct product of two complete graphs $K_n$ (also known as the diagonal $2$-dimensional Hamming graphs), which shows that such a product is also challenging as well as it is the Cartesian case. In the next section we give some basic definitions; compute the value of $\kappa(G)$ when $G$ is a direct product; and present some tools that we shall later need in our exposition. The remaining sections of the article contains the formulae, and one bound, for this mentioned value. Since the case $k=1$ coincides with the classical metric dimension of graphs and the metric dimension of the direct product of two complete graphs is already known from \cite{kuziak17}, along the work, we only consider $\wdim_{k}(G)$ for any suitable value $k\ge 2$. In Section \ref{sec:small}, we compute the value of this parameter when $k\in \{2,3,4\}$. Section \ref{sec:large} is focused on the larger values of $k$ from the set $2n-1,2n,2n+1,2n+2$, where $n$ is the order of the complete graph $K_n$ used in the direct product. Section \ref{sec:general} considers the remaining situations for the parameter since they can be all treated in a similar manner.

\section{Preliminaries}

Given two graphs $G$ and $H$, the \emph{direct product graph} $G\times H$ is the graph with vertex set $V(G)\times V(H)$, where two vertices $(g,h),(g',h')$ are adjacent if $gg'\in E(G)$ and $hh'\in E(H)$. In order to simplify our notation, given an integer $t\ge 1$, throughout this paper we write $[t]=\{1,\dots,t\}$. Also, we assume along the exposition that $V = V(K_n \times K_n)=[n]\times [n]$ to denote the vertex set of $K_n \times K_n$. 

Considering the direct product graph $K_n \times K_n$, for $i,j \in [n]$, we define the two sets $$L_i = \{(i,j) : j \in [n]\} \quad \text{and} \quad L^j = \{(i,j) : i \in [n]\}$$ as \textit{vertical} and \textit{horizontal layers} with respect to the vertex $(i,j)\in V(K_n \times K_n)$, respectively. Also, let $L_i^j = L_i \cup L^j$ denote its \textit{intersecting layers} (of the vertex $(i,j)$). 
Two vertices $x,y \in V$ are said to \textit{lie in the same layer} if they share a horizontal layer or a vertical layer; otherwise, they \textit{lie in different layers}. 
For distinct $x,y \in V$, and based on the classical definition of the direct product of graphs, we readily observe that
\begin{equation*}
    d(x,y) =
    \begin{cases}
        1, &\text{if $x,y$ lie in different layers}; \\
        2, &\text{otherwise}.
    \end{cases}
\end{equation*}

From this distance formula, the following properties follow immediately.
    Let $n \geq 3$ be an integer and $x,y,s \in V$ with $x = (i,j)$ and $y = (i',j')$. If $x$ and $y$ lie in the same layer, then
    \begin{equation}
    \label{eq: delta x,y same layer}
        \Delta_s(x,y) =
        \begin{cases}
            2, &s \in \{x,y\}; \\
            1, &s \in L_{i}^{j} \diff L_{i'}^{j'}; \\
            0, &\text{otherwise},
        \end{cases}
    \end{equation}
    where $|L_{i}^{j} \diff L_{i'}^{j'}| = 2n-2$ (recall that $\diff$ represents the symmetric difference between two sets). If $x$ and $y$ lie in different layers, then
    \begin{equation}
    \label{eq: delta x,y different layers}
        \Delta_s(x,y) = 
        \begin{cases}
            1, &s \in (L_i^j \cup L_{i'}^{j'}) \setminus \{(i,j'), (i',j)\}; \\
            0, &\text{otherwise},
        \end{cases}
    \end{equation}
    and there are $4n-6$ vertices in the first case. These observations lead to the following result.

\begin{theorem}
\label{thm: kappa}
    For every integer $n \geq 3$, 
    \begin{equation*}
    \kappa(K_n \times K_n) =
    \begin{cases}
        6, &n=3; \\
        2n+2, &n \geq 4.
    \end{cases}
    \end{equation*}
\end{theorem}

\begin{proof}
    Let $x,y \in V$ be distinct. If $x,y$ lie in the same layer, then from \eqref{eq: delta x,y same layer}, we have $\Delta(x,y) = 2(2) + 2(n-1)(1) = 2n+2$. Otherwise, $\Delta(x,y) = 4n-6$. Hence, using \eqref{eq-kappa}, we have $\kappa(K_n \times K_n) = \min\{2n+2,4n-6\}$ which yields the result.
\end{proof}

Based on this result, we may compute $\wdim_k(K_n \times K_n)$ for each integer $k\in [6]$ if $n=3$ or $k\in [2n+2]$ if $n\ge 4$. To determine an upper bound of $\wdim_k(K_n \times K_n)$, we construct a weak $k$-resolving set $S$ where we need to verify that every pair of vertices $x,y \in V$ satisfies $\Delta_{S}(x,y) \ge k$. To this end, we require the following observation, which demonstrates a way to calculate $\Delta_S(x,y)$. In these computations (and along the whole exposition), we say that a vertex $w\in S$ contributes (a quantity-number) to $\Delta_S(x,y)$ if $\Delta_w(x,y)\ge 1$. Clearly, if $\Delta_w(x,y)=0$, then $w$ does not contribute to $\Delta_S(x,y)$.

\begin{proposition}
\label{prop: delta formula}
Let $n\ge 4$, $\emptyset \neq S\subseteq V$, and take distinct vertices $x=(i,j), y=(i',j').$
Let $a_r=|L_r\cap S|$ and $b_s=|L^s\cap S|$ for $r\in\{i,i'\}$ and $s\in\{j,j'\}$.
\begin{enumerate}
  \item If $x,y$ lie in the same horizontal layer, then 
  $$\Delta_S(x,y)=a_i+a_{i'}+|\{x,y\}\cap S|.$$
  \item If $x,y$ lie in the same vertical layer, then
  $$\Delta_S(x,y)=b_j+b_{j'}+|\{x,y\}\cap S|.$$
  \item If $x,y$ lie in different layers and set $z_1=(i',j)$ and $z_2=(i,j')$, then
  $$\Delta_S(x,y)=a_i+a_{i'}+b_j+b_{j'}-|\{x,y\}\cap S|-2|\{z_1,z_2\}\cap S|.$$
\end{enumerate}
\end{proposition}

\begin{proof}
All three formulas follow from the same principle: start by counting the elements of $S$ in the relevant horizontal and vertical layers, then correct the double-counting for corner vertices: $x,y,z_1,z_2$.

For (1), only vertical layers $L_i,L_{i'}$ might contribute. Every element of $(L_i\cup L_{i'})\cap S$ different from $x,y$ contributes $1$ to $\Delta_S(x,y)$. Moreover, $x$ or $y$ might contribute $2$, when $S$ includes $x$ or $y$, respectively. Thus
$$\Delta_S(x,y)=|(L_i\cup L_{i'})\cap S|+|\{x,y\}\cap S|=a_i+a_{i'}+|\{x,y\}\cap S|.$$
A similar argument also applies to the second case due to the symmetry of $K_n \times K_n$.

For (3), let $R=(L_i\cup L_{i'}\cup L^j\cup L^{j'}) \setminus \{x,y,z_1,z_2\}$ and set $c:=a_i+a_{i'}+b_j+b_{j'}.$
The sum $c$ counts every non-corner element of $S\cap R$ exactly once and counts each corner in $\{x,y,z_1,z_2\}\cap S$ twice. Thus
$$c=|R \cap S|+2|\{x,y,z_1,z_2\}\cap S|.$$
On the other hand, $\Delta_S(x,y)$ counts each non-corner element of $S\cap R$ once, counts each of $x,y$ (if in $S$) once more, and counts each of $z_1,z_2$ (if in $S$) zero extra times. Hence
$$\Delta_S(x,y)=|R \cap S|+|\{x,y\}\cap S|.$$
Eliminating $|R \cap S|$ using the identity for $c$ gives
$$\Delta_S(x,y)=c-|\{x,y\}\cap S|-2|\{z_1,z_2\}\cap S|,$$
which is the displayed formula. The cases $i>i',\,j>j'$ (or other index orderings) follow by relabeling the indices due to the symmetry of $K_n \times K_n$.
\end{proof}

With the tools above in hand, we are now in the position to compute the weak $k$-metric dimension of $K_n\times K_n$ for any suitable $k\in [\kappa(K_n\times K_n)]$. First, by computer search, we have obtained the following values when $n=3$.

$$\wdim_k(K_3 \times K_3)=\left\{\begin{array}{ll}
    4, & k = 2; \\
    6, & k = 3, 4; \\
    8, & k = 5; \\
    9, & k = 6.
\end{array}
\right.$$

In this sense, from now on, we focus on the cases $n\ge 4$. This is made in the following sections.

\section{Smaller values of $k$}\label{sec:small}

In this section, we consider the cases $k\in\{2,3,4\}$ since they behave differently from the other larger values of $k$, and indeed require more technical arguments. 

\begin{theorem}
    For every integer $n \ge 4$, $\wdim_2(K_n \times K_n) = n + \ceil{\frac{n}{3}}$.
\end{theorem}

\begin{proof}
    We first define $$S_0 = \{(3i,3i),(3i,3i+1),(3i+1,3i+2),(3i+2,3i+2) : i \in \left[\floor{n/3}\right]\},$$ and we claim that the set
    \begin{align*}
        S =
        \begin{cases}
            S_0, & n \equiv 0 \pmod{3}; \\
            S_0 \cup \{(n,n-1),(n,n)\}, & n \equiv 1 \pmod{3}; \\
            S_0 \cup \{(n-1,n-1),(n-1,n),(n,n-2)\}, & n \equiv 2 \pmod{3}
        \end{cases}
    \end{align*}
constitutes a weak $2$-resolving set of size $2n-2$. Notice that $|S|=n + \ceil{\frac{n}{3}}$. See Fig. \ref{fig: k = 2} for the set $S$ in $K_n \times K_n$ for $n = 6,7,8$, respectively. Here, we do not draw the edges for a clearer construction.

We note first that our construction of $S$ ensures that each layer, horizontal and vertical, has a vertex in $S$. Take any pair $x=(i,j),y=(i',j') \in V$. If both vertices lie in the same (horizontal) layer, then our construction implies $\Delta_S(x,y) \ge 2$, by Proposition \ref{prop: delta formula}. Now, let us assume that $x,y$ lie in different layers and consider $z_1 = (i',j)$. Whether $z_1 \in S$ or $z_1 \notin S$, our construction ensures the existence of another element of $S$ in $L_{i'}^j$ that contributes $1$ to the sum $\Delta_S(x,y)$. The same applies also for $z_2 = (i,j')$. Hence, $\Delta_S(x,y) \geq 2$. Thus, $S$ is a weak $2$-resolving set, and so, $\wdim_2(K_n \times K_n) \le n + \ceil{\frac{n}{3}}$.

\medskip
For the lower bound, let $S$ be a weak $2$-metric basis. We first claim that each vertical and horizontal layer must have at least one element in $S$. Suppose that a layer $L_i$ has no element in $S$. For every $i' \neq i$, take any $j \in [n]$ such that $(i',j) \notin S$. Then we must have $|L_{i'} \cap S| \geq 2$ for it to satisfy $\Delta_S((i,j),(i',j)) \ge 2$, but this implies $|S| \ge \sum_{i' \ne i} |L_{i'} \cap S| \ge 2(n-1) > n + \ceil{n/3}$ since $n\ge 4$, contradicting our previous upper bound.

In light of this claim, we locally say that a layer is \textit{light} if it contains exactly one vertex in $S$, and we call it \textit{heavy} if otherwise. Let $a$ and $b$ denote the number of light vertical and light horizontal layers, respectively. Since each heavy vertical layer has at least two vertices in $S$, we have $|S| \ge a + 2(n-a) = 2n - a$, which leads to $a \ge 2n - |S|$. Similarly, we obtain $b \ge 2n - |S|$. Now, the total number of vertices from $S$, lying in all light vertical and light horizontal layers is $a + b - |I|$ where $I \subseteq S$ is the set of vertices in $S$ that are the unique element in both their vertical and horizontal layers. Thus, it holds that $|S| \ge a + b - |I|$.

We claim that $I = \emptyset$. Suppose there is a vertex $(i,j)$ such that it is the only vertex in both its vertical and horizontal layers. For every $i' \ne i$, take any $j' \ne j$ such that $(i',j') \notin S$. For $\Delta_S((i,j'),(i',j')) \ge 2$, we must have $|L_{i'} \cap S| \ge 2$. But again, this implies $|S| \ge 1 + 2(n-1) = 2n - 1 > n + \ceil{n/3}$, a contradiction. Thus, the set $I$ is empty as claimed, and our set $S$ has cardinality $$|S| \ge a + b = (2n - |S|) + (2n - |S|) = 4n - 2|S|,$$ and so, $|S| \ge 4n/3$. Since $|S|$ is an integer, $|S| \ge \ceil{4n/3} = n + \ceil{n/3}$. This completes the proof.
\end{proof}

\begin{figure}
    \centering
    \begin{subfigure}{0.3\textwidth}
    \centering
    \newcommand{\n}{6}
    \begin{tikzpicture}[scale=0.45]
        \foreach \i in {1,...,\n}
            \foreach \j in {1,...,\n}
                \node[circle, draw, inner sep=1.8pt] (\i-\j) at (\i,\j) {};
        
        \foreach \p in {1-1,1-2,2-3,3-3,4-4,4-5,5-6,6-6}
            \node[circle,fill=black,inner sep=1.8pt] at (\p) {};
    \end{tikzpicture}
    \end{subfigure}
    \hfill
    \begin{subfigure}{0.3\textwidth}
    \centering
    \newcommand{\n}{7}
    \begin{tikzpicture}[scale=0.45]
        \foreach \i in {1,...,\n}
            \foreach \j in {1,...,\n}
                \node[circle, draw, inner sep=1.8pt] (\i-\j) at (\i,\j) {};
        
        \foreach \p in {1-1,1-2,2-3,3-3,4-4,4-5,5-6,6-6,7-6,7-7}
            \node[circle,fill=black,inner sep=1.8pt] at (\p) {};
    \end{tikzpicture}
    \end{subfigure}
    \hfill
    \begin{subfigure}{0.3\textwidth}
    \centering
    \newcommand{\n}{8}
    \begin{tikzpicture}[scale=0.45]
        \foreach \i in {1,...,\n}
            \foreach \j in {1,...,\n}
                \node[circle, draw, inner sep=1.8pt] (\i-\j) at (\i,\j) {};
        
        \foreach \p in {1-1,1-2,2-3,3-3,4-4,4-5,5-6,6-6,7-7,7-8,8-8}
            \node[circle,fill=black,inner sep=1.8pt] at (\p) {};
    \end{tikzpicture}
    \end{subfigure}
    \hfill
    
    \caption{The weak 2-resolving basis of $K_n \times K_n$ for $n = 6, 7, 8$, respectively}
    
    \label{fig: k = 2}
\end{figure}
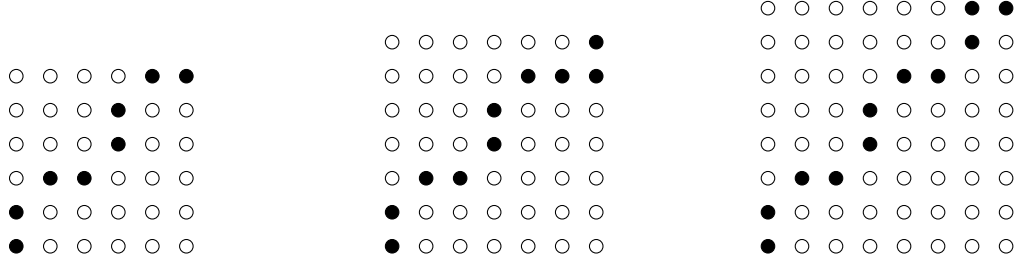

\begin{theorem}
    For every integer $n \ge 4$, $\wdim_3(K_n \times K_n) = 2n$.
\end{theorem}

\begin{proof}
    We first establish the upper bound. We define
    \begin{align*}
        D_0 &= \{(i,i) : i \in [n]\}, \\ 
        D_1 &= \{(i+1,i) : i \in [n-1]\} \cup \{(1,n)\},
    \end{align*}
    and we claim that the set $S = D_0 \cup D_1$
    is a weak $3$-resolving set of size $2n$. 
    Consider two vertices $x = (i,j)$ and $y = (i',j)$ lying in the same horizontal layer. The vertical case behaves similarly by the symmetry of $K_n \times K_n$.
    The structure of $S$ implies that each layer $L_i$ and $L_{i'}$ contains exactly two vertices in $S$, where Proposition \ref{prop: delta formula} immediately implies $\Delta_S(x,y) \ge 3$. Now, we consider the pair $x = (i,j), y = (i',j')$ that belongs to different layers, and set $z_1 = (i',j)$ and $z_2 = (i,j')$. Proposition \ref{prop: delta formula} implies
    \begin{equation*}
        \Delta_S(x,y) \ge 8 - |\{x,y\}\cap S| - 2|\{z_1,z_2\}\cap S|.
    \end{equation*}
    Our construction of $S$ ensures that we cannot have $|\{x,y\} \cap S| = 2 = |\{z_1,z_2\} \cap S|$, so either we have $|\{x,y\} \cap S| \le 1$ and $|\{z_1,z_2\} \cap S| \le 2$, or $|\{x,y\} \cap S| \le 2$ and $|\{z_1,z_2\} \cap S| \le 1$. Both cases imply $\Delta_S(x,y) \ge 3$. Therefore, $S$ is a weak $3$-resolving set, hence $\wdim_3(K_n \times K_n) \le 2n$.

    We now prove the lower bound. Let $S$ be a weak $3$-metric basis. We first claim that every layer, both horizontal and vertical, contains an element of $S$. Suppose that $L_1$, without loss of generality, has no element of $S$. Then we must have $|L_i \cap S| \ge 3$ for every $i \ne 1$ so that $\Delta_S((1,1),(i,1)) \ge 3$. But then $|S| \ge 3(n-1) > 2n$, since $n\ge 4$, which contradicts $\wdim_3(K_n \times K_n) \le 2n$. Now, let us assume that $S$ has a size of at most $2n-1$. This ensures the existence of horizontal and vertical layers with exactly one element of $S$, say $L_i \cap S = \{(i,j)\}$ and $L^{j'} \cap S = \{(i',j')\}$. Additionally, $|L_{i'} \cap S| = 2 = |L_{j''} \cap S|$ for all $i' \neq i$ and $j'' \neq j'$. If $(i,j) = (i',j')$, then we choose any vertex $(i'',j'') \in S$ in $V \setminus (L_i \cup L^{j})$, and thus $\Delta_S((i,j''),(i'',j)) = 2$. Otherwise, we must have $i \ne i'$ and $j \ne j'$, and $\Delta_S((i,j'),(i',j)) = 2$. Both cases contradict that $S$ is a  weak $3$-resolving set. Therefore, $\wdim_3(K_n \times K_n) \ge 2n$, and we are done.
\end{proof}

The next result shows the only case in our investigation in which an exact formula has not been deduced. 

\begin{theorem}\label{th:k-4}
    For every integer $n \ge 9$, $\wdim_4(K_n \times K_n) \le 2n + 1 + \floor{\frac{n}{4}}$.
\end{theorem}

\begin{proof}
To prove such upper bound, we define the set
    \begin{align*}
        S 
        &= \{(i,i) : i \in [n]\} \cup \{(i,i+2) : i \in [n-2]\} ~\cup \\
        &\phantom{==} \{(2,1),(3,2),(n-1,n-2),(n,n-1)\} ~\cup \\
        &\phantom{==} \{(i,i-1) : i \equiv 2 \pmod{4}, \; 6 \leq i \leq n-2\}
    \end{align*}
    and claim that $S$ is a weak $4$-resolving set of size $2n+1+\floor{\frac{n}{4}}$. See Fig. \ref{fig: k = 4} for an illustration of the set $S$ for $n = 13$ and $14$. Similarly to the previous result, the structure of $S$ asserts that each layer contains at least two vertices in $S$. Furthermore, $|L_i \cap S| = 3$ if and only if $i \in \{2,3\}$ or $i \equiv 2 \pmod{4}$, $6 \le i \le n-2$, and $|L^j \cap S| = 3$ if and only if $j \in \{n-2,n-1\}$ or $j \equiv 1 \pmod{4}$, $5 \le j \le n-3$.

    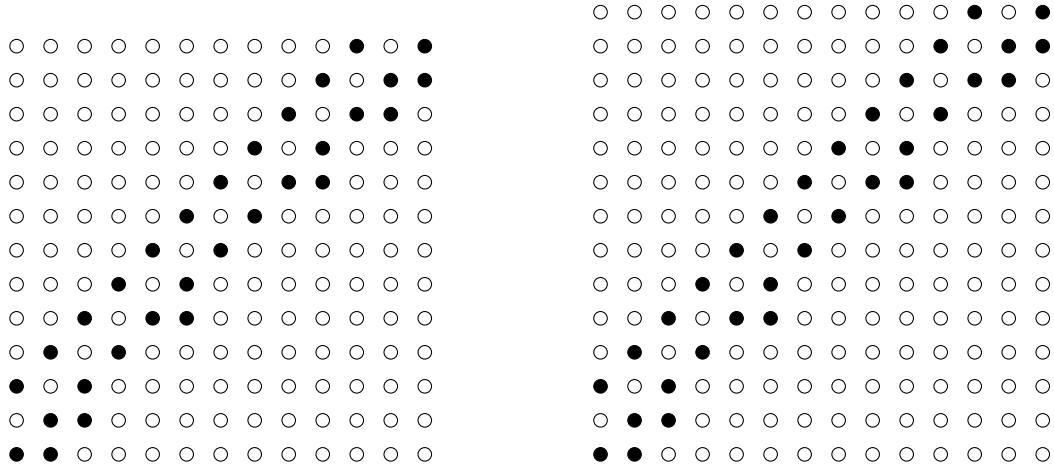
\begin{figure}
    \centering
    \begin{subfigure}{0.47\textwidth}
    \centering
    \newcommand{\n}{13}
    \begin{tikzpicture}[scale=0.45]
        \foreach \i in {1,...,\n}
            \foreach \j in {1,...,\n}
                \node[circle, draw, inner sep=1.8pt] (\i-\j) at (\i,\j) {};
        
        \foreach \p in {
            1-1,2-2,3-3,4-4,5-5,6-6,7-7,8-8,9-9,10-10,11-11,12-12,13-13,
            1-3,2-4,3-5,4-6,5-7,6-8,7-9,8-10,9-11,10-12,11-13,
            2-1,3-2,6-5,10-9,12-11,13-12}
            \node[circle,fill=black,inner sep=1.8pt] at (\p) {};
    \end{tikzpicture}
    \end{subfigure}
    \hfill
    \begin{subfigure}{0.47\textwidth}
    \centering
    \newcommand{\n}{14}
    \begin{tikzpicture}[scale=0.45]
        \foreach \i in {1,...,\n}
            \foreach \j in {1,...,\n}
                \node[circle, draw, inner sep=1.8pt] (\i-\j) at (\i,\j) {};
        
        \foreach \p in {
            1-1,2-2,3-3,4-4,5-5,6-6,7-7,8-8,9-9,10-10,11-11,12-12,13-13,14-14,
            1-3,2-4,3-5,4-6,5-7,6-8,7-9,8-10,9-11,10-12,11-13,12-14,
            2-1,3-2,6-5,10-9,13-12,14-13}
            \node[circle,fill=black,inner sep=1.8pt] at (\p) {};
    \end{tikzpicture}
    \end{subfigure}
    
    \caption{The weak 4-resolving set $S$ of $K_n \times K_n$ for $n = 13$ and $14$, respectively}
    
    \label{fig: k = 4}
\end{figure}
    
    For any two vertices $x=(i,j),y=(i',j')$, if they belong to the same layer, then $\Delta_S(x,y) \ge 2+2=4$ immediately by Proposition \ref{prop: delta formula}. Let us now assume that they belong to different layers, and let $z_1=(i,j')$, $z_2=(i',j)$. Together, $\{x,y,z_1,z_2\}$ forms a rectangle defined by $x,y$ with each element being its corners. We define $\alpha := |\{x,y\} \cap S|$ and $\beta := |\{z_1,z_2\} \cap S|$. Proposition \ref{prop: delta formula} gives $$\Delta_S(x,y) \geq 8 - \alpha - 2\beta.$$ This implies $\Delta_S(x,y) \ge 4$ unless $(\alpha,\beta) \in \{(1,2),(2,2)\}$. We analyze these two cases. One may verify by the structure of $S$ that there is no $x,y$ whose rectangle has four corners in $S$, that is, the case $(\alpha,\beta)=(2,2)$ is impossible. 
    
    We now assume that $(\alpha,\beta)=(1,2)$, that is, we consider the rectangles having only three corners in $S$. It is sufficient to show that at least one of the layers involved (which are $L_i,L_{i'},L^j,L^{j'}$) has three elements in $S$ (note that if such property is true, then Proposition \ref{prop: delta formula} implies $\Delta_S(x,y) \ge (3+2+2+2) - \alpha - 2\beta = 4$, and the proof is complete). Assume without loss of generality that $i < i'$. We analyze cases based on the value of $i$ where $1 \le i \le n-1$. If $|L_i \cap S| = 3$, then we are done, hence we assume otherwise. If $i=1$, then there are two possibilities for $\{x,y\}$: $\{(1,2),(2,1)\}$ and $\{(1,3),(3,1)\}$, where both involve $L_2$ and $L_3$, respectively, and each contain three elements in $S$. If $i = n-1$, then it must be $\{x,y\}=\{(n-1,n-1),(n,n-2)\}$, and we have $|L^{n-1} \cap S| = 3$. It remains to consider $i \in [6,n-2]$, $i \not\equiv 2 \pmod{4}$.

    \medskip
    \noindent
    \textbf{Case 1:} If $i \equiv 0 \pmod{4}$, then there is only one possibility for $\{x,y\}$: $\{(i,i+2),(i+2,i)\}$, and we have $|L_{i+2} \cap S| = 3$.

    \medskip
    \noindent
    \textbf{Case 2:} If $i \equiv 1 \pmod{4}$, then $(i,i) \in S$ serves as a corner for every possibilities of $\{x,y\}$, and we have $|L^i \cap S| = 3$ immediately.

    \medskip
    \noindent
    \textbf{Case 3:} If $i \equiv 3 \pmod{4}$, then we have two pairs of $\{x,y\}$: $\{(i,i+2),(i+2,i)\}$ or $\{(i,i+2),(i+3,i)\}$, and both involve $L^{i+2}$ which contains three elements in $S$.

    \medskip
    \noindent
    Thus, our previous claim follows, and we have that $\wdim_4(K_n \times K_n) \le 2n+1+\floor{\frac{n}{4}}$.
\end{proof}

The result above clearly leads to a question concerning the exact value of $\wdim_4(K_n \times K_n)$ when $n\ge 4$.

\section{Larger values of $k$}\label{sec:large}

We now consider in this section, the largest values of $k$ and one particular case, where the techniques used are rather simple and require less effort.

\begin{theorem}
For every integer $n \geq 4$,
\begin{equation*}
    \wdim_k(K_n \times K_n) =
    \begin{cases}
        n^2, & k \in \{2n+1,2n+2\}; \\
        n^2-1, & k = 2n; \\
        n^2-n, & k = 2n-1, \quad n \ge 5; \\
        13, & (n,k) = (4,7).
    \end{cases}
\end{equation*}
\end{theorem}

\begin{proof}
We analyze each case separately, depending on the corresponding values of $k$.

\medskip
\noindent
\textbf{Case 1:} $k \in \{2n+1,2n+2\}$. Let $S$ be a weak $k$-resolving set of $K_n \times K_n$ where $k \in \{2n+1, 2n+2\}$. Suppose there exists $x = (i,j) \in V \setminus S$. Then, for any $y  \in L_i^j \setminus \{x\}$, 
$$\Delta_S(x,y) = \Delta(x,y) - \Delta_x(x,y) = (2n+2) - 2 = 2n,$$ a contradiction. Thus, $S = V$ and $\wdim_k(K_n \times K_n) = n^2$.

\medskip
\noindent
\textbf{Case 2:} $k = 2n$. Let $S = V \setminus \{(1,1)\}$. Since $\Delta_s(x,y) \leq 2$ for any triple $x,y,s \in V$, Theorem \ref{thm: kappa} gives
$$\Delta_S(x,y) = \Delta(x,y) - \Delta_{(1,1)}(x,y) \geq (2n+2) - 2 = 2n.$$ 
Hence, $S$ is a weak $2n$-resolving set, and so $\wdim_{2n} ( K_n \times K_n) \leq n^2-1$. Now, suppose that there is a weak $2n$-resolving set $S$ of size at most $n^2-2$. Take any pair $(i,j),(i',j') \in V \setminus S$. If the two vertices lie in the same layer, then $\Delta_S((i,j),(i',j')) \leq 2n-2 < 2n$. On the contrary, if the two vertices lie in different layers, then $\Delta_S((i,j),(i,j')) \leq 2n-1 < 2n$. Both situations lead to a contradiction. Therefore, $\wdim_{2n}(K_n \times K_n) \geq n^2-1$.

\medskip
\noindent
\textbf{Case 3:} $k = 2n-1$. The case $(n,k) = (4,7)$ can be easily found by computer, where the set $S = V \setminus \{(1,2),(2,3),(3,4)\}$ serves as its weak $k$-metric basis. Let $n \ge 5$ and $S = V \setminus D$ where $D = \{(i,i) : i \in [n]\}$. Hence, by construction, each layer has exactly $n-1$ vertices in $S$. Let $x,y \in V$ be distinct. We first assume that $x,y$ lie in the same layer. Our construction implies at least one of $x,y$ is in $S$. Proposition \ref{prop: delta formula} implies $$\Delta_S(x,y) = a_i + a_{i'} + |\{x,y\} \cap S| \ge 2(n-1) + 1 = 2n-1.$$
Now let us assume that $x,y$ lie in different layers. Proposition \ref{prop: delta formula} implies
$$\Delta_S(x,y)=4(n-1)-|\{x,y\}\cap S|-2|\{z_1,z_2\}\cap S| \geq 4(n-1)-2-2(2) \ge 2n-1$$ since $n \ge 5$.
Thus, $S$ is a weak $(2n-1)$-resolving set, and so $\wdim_{2n-1}(K_n \times K_n) \leq n^2-n$. If instead $|S| \le n^2 - n - 1$, then some layer contains at least two vertices outside $S$, giving $\Delta_S(x,y) \le 2n-2$, a contradiction. Therefore, $\wdim_{2n-1}(K_n \times K_n) = n^2-n$.
\end{proof}

\section{The general case}\label{sec:general}

Once these particular situations have been managed in the two sections above, we may present the general contribution for the remaining cases of $k$, which is addressed next.

\begin{theorem}
\label{thm: main result}
    For every integer $n \geq 4$ and $1 \leq t \leq n-2$,
    \begin{equation*}
        \wdim_k(K_n \times K_n) =
        \begin{cases}
            n^2-n-1, & k = 2n-2; \\
            n^2-tn, & k = 2n-2t, \quad t \geq 2; \\
            n^2-tn-\floor{\frac{n}{t+1}}, & k = 2n-2t-1.
        \end{cases}
    \end{equation*}
\end{theorem}

The proof of this result requires several deductions, and a few technical analysis. In this sense, we next present such proof into two separated subsections, each of them considering a situation depending on the parity of the integer $k$. 

\subsection{Even case for $k$}

We first determine the weak $k$-metric dimension of $K_n \times K_n$ for even integers $6\le k \leq 2n-2$. For that end, given a weak $k$-resolving set of vertices $S$, we need the following observation which bounds the number of ``holes'' of $S$, which are understood as the elements of $\overline{S} := V \setminus S$, in each layer of $K_n \times K_n$.

\begin{lemma}
\label{lem: <= 2t+1}
    Let $n \geq 4$ and $t \in [n-1]$ be integers, and let $S$ be a weak $(2n-2t)$-resolving set of $K_n \times K_n$. For every $i,i' \in [n]$,
    $$|(L_i \cup L_{i'}) \cap \overline{S}| \leq
    \begin{cases}
        2t,     &\text{if $(i,j) \in \overline{S}$ and $(i',j) \in \overline{S}$ for some $j$};\\
        2t+1, &\text{if otherwise}.
    \end{cases}
    $$
\end{lemma}

\begin{proof}
    Let us first assume that there is some $j$ such that $x = (i,j) \in \overline{S}$ and $y =(i',j) \in \overline{S}$. If $|(L_i \cup L_{i'}) \cap \overline{S}| \geq 2t+1$, then Proposition \ref{prop: delta formula} implies $\Delta_{\overline{S}}(x,y) \geq (2t+1) + 2 = 2t+3$, and thus $\Delta_S(x,y) \leq (2n+2) - (2t+3) = 2n - 2t - 1$, contradicting that $S$ is a weak $(2n-2t)$-resolving set.
    
    Let us assume otherwise that such $j$ does not exist. Suppose that $|(L_i \cup L_{i'}) \cap \overline{S}| \geq 2t+2$. Take any $x = (i,j') \in \overline{S}$ and $y = (i',j') \notin \overline{S}$.
    Then Proposition \ref{prop: delta formula} implies $\Delta_{\overline{S}}(x,y) \geq (2t+2)+1 = 2t+3$, and so $\Delta_S(x,y) \leq (2n+2) - (2t+3) = 2n -2t -1$, which is a contradiction.
\end{proof}

Our arguments require that we consider separately the case $k=2n-2$. The remaining ones are further on addressed altogether.

\begin{theorem}
\label{thm: 2n-2}
    For every integer $n \geq 4$, $$\wdim_{2n-2}(K_n \times K_n) = n^2-n-1.$$
\end{theorem}

\begin{proof}
    For the upper bound, let $S_0 = \{(1,1),(n,n)\} \cup \{(i,i+1) : i \in [n-1]\}$. We claim that the set $S = V \setminus S_0$ is a weak $(2n-2)$-resolving set of $K_n \times K_n$. 
    Take any two vertices $x = (i,j)$ and $y = (i',j)$ lying in the same (horizontal) layer. By construction, $|L_1 \cap S| = n-2$ and $|L_{i} \cap S| = n-1$ for every $i \neq 1$.
    Notice that $x \notin S$ and $y \notin S$ if and only if $x = (n-1,n)$ and $y = (n,n)$. Then, all other pairs $x,y$ must have $x \in S$ or $y \in S$, and at most one of them may be in $L_1$. Proposition \ref{prop: delta formula} implies $$\Delta_S(x,y) \ge (n-2)+(n-1)+1=2n-2.$$ 
    If $x = (n-1,n)$ and $y = (n,n)$, then we have $\Delta_{S}(x,y) = (n-1)+(n-1) = 2n-2$. Thus in all cases, it holds $\Delta_S(x,y) \ge 2n-2$. Hence, $S$ is a weak $(2n-2)$-resolving set and $\wdim_k(K_n \times K_n) \le n^2-n-1$.

\medskip
    We now prove the lower bound. Assume to the contrary that there is a weak $(2n-2)$-resolving set $S$ with $|S| \leq n^2-n-2$, which means $|\overline{S}| \geq n+2$. Since there are $n$ vertical layers, there must be a layer $L_i$ such that $a_i := |L_i \cap \overline{S}| \geq 2$. Lemma \ref{lem: <= 2t+1} ensures, letting $t=1$, that $a_{i'} := |L_{i'} \cap \overline{S}| \leq 1$ for every $i' \neq i$, which means $\sum_{i' \neq i} a_{i'} \leq n-1$. Then $$a_i = |\overline{S}| - \sum_{i'\neq i} a_{i'} \geq (n+2) - (n-1) = 3.$$ However, since $a_i \geq 3$, Lemma \ref{lem: <= 2t+1} again implies $a_{i'} = 0$ for every $i' \neq i$. Then $$a_i = |\overline{S}| - \sum_{i' \neq i} a_{i'} \geq n+2 \geq 6$$ since $n \geq 4$, which immediately contradicts Lemma \ref{lem: <= 2t+1}. Thus, there is no weak $(2n-2)$-resolving set of size at most $n^2-n-2$, and therefore $\wdim_{2n-2}(K_n \times K_n) \geq n^2-n-1$ as desired.
\end{proof}

\begin{theorem}
\label{thm: wdim even k}
    For every integer $n \geq 4$ and $2 \leq t \leq n-3$, $$\wdim_{2n-2t}(K_n \times K_n) = n^2-tn.$$
\end{theorem}

\begin{proof}
    We first establish the upper bound by defining
    \begin{align*}
        D_0 &= \{(i,i) : i \in [n]\}, \\ 
        D_j &= \{(i+j,i) : i \in [n-j]\} \cup \{(i,n-j+i) : i \in [j]\} \quad (1 \leq j \leq t-1).
    \end{align*}
    We claim that the set $S = V \setminus \bigcup_{j=0}^{t-1} D_j$
    is a weak $(2n-2t)$-resolving set. Clearly, $|S| = n^2 - tn$. See Fig. \ref{fig: even k} for illustration.
    
    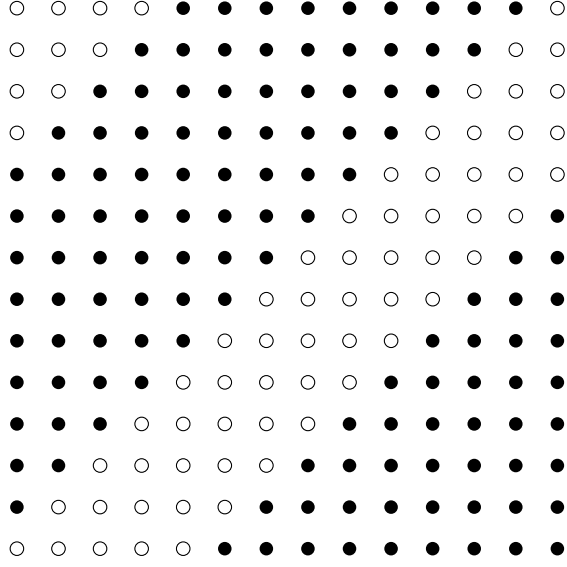
\begin{figure}
        \centering
        \newcommand{\n}{14}
        
        \begin{tikzpicture}[scale=0.55]
            \foreach \i in {1,...,\n}
                \foreach \j in {1,...,\n}
                    \node[circle,fill=black,inner sep=1.8pt] (\i-\j) at (\i,\j) {};
            
            \foreach \p in {
                1-1,2-2,3-3,4-4,5-5,6-6,7-7,8-8,9-9,10-10,11-11,12-12,13-13,14-14,
                2-1,3-2,4-3,5-4,6-5,7-6,8-7,9-8,10-9,11-10,12-11,13-12,14-13,
                3-1,4-2,5-3,6-4,7-5,8-6,9-7,10-8,11-9,12-10,13-11,14-12,
                4-1,5-2,6-3,7-4,8-5,9-6,10-7,11-8,12-9,13-10,14-11,
                5-1,6-2,7-3,8-4,9-5,10-6,11-7,12-8,13-9,14-10,
                1-14,1-13,1-12,1-11,2-14,2-13,2-12,3-14,3-13,4-14}
                \node[circle,draw,fill=white,inner sep=1.8pt] at (\p) {};
        \end{tikzpicture}
        
        \caption{The weak $(2n-2t)$-resolving set $S$ of $K_n \times K_n$ for $(n,t) = (14,5)$}
        
        \label{fig: even k}
    \end{figure}
    
    Consider two vertices $x = (i,j)$ and $y = (i',j)$ lying in the same horizontal layer. The vertical case behaves similarly by the symmetry of $K_n \times K_n$.
    By our construction, each layer $L_i$ and $L_{i'}$ contains exactly $n-t$ vertices from $S$. Proposition \ref{prop: delta formula} implies
    $$\Delta_S(x,y) = 2(n-t) + |\{x,y\} \cap S| \ge 2(n-t)+0=2n-2t.$$

    Now, consider the pair of vertices $x = (i,j), y = (i',j')$ that lie in different layers, and set $z_1 = (i',j)$ and $z_2 = (i,j')$. Since $t \le n-3$, Proposition \ref{prop: delta formula} implies
    $$\Delta_S(x,y) = 4(n-t) - |\{x,y\}\cap S| - 2|\{z_1,z_2\}\cap S| \ge 4(n-t) - 6 \ge 2n-2t.$$
    
    We now establish the lower bound. Suppose that there is a weak $(2n-2t)$-resolving set $S$ with $|S| \leq n^2 - tn-1$, or equivalently, with $|\overline{S}|\geq tn+1$. Similarly to the proof of Theorem \ref{thm: 2n-2}, there must exist a layer $L_i$ such that $a_i = |L_i \cap \overline{S}| \geq t+1$. Lemma \ref{lem: <= 2t+1} ensures that $a_{i'} = |L_{i'} \cap \overline{S}| \leq t$ for every $i' \neq i$. We first claim that $a_{i'} = t$ for every $i' \neq i$. Suppose otherwise that $a_{i''} \le t-1$ for some $i'' \neq i$. Then, we have $$a_i = |\overline{S}| - \sum_{i' \neq i} a_{i'} \ge (tn+1) - (t-1) - t(n-2) = t + 2.$$ Once more, Lemma \ref{lem: <= 2t+1} implies $a_{i'} \leq t-1$ for every $i' \neq i$. However, this implies $$a_i = |\overline{S}| - \sum_{i' \neq i} a_{i'} \geq (tn+1) - (t-1)(n-1) = t + n \geq 2t + 2$$ since $t \leq n-2$, and this contradicts Lemma \ref{lem: <= 2t+1}. Thus, we must have $a_i \geq t + 1$ and $a_{i'} = t$ for every $i' \neq i$.

    By the symmetry of $K_n \times K_n$, this configuration of $a_i$s must also be satisfied by the horizontal layers. In other words, each horizontal layer must contain at least $t \geq 2$ holes. Now, take any $j \in [n]$ such that $x = (i,j) \in L_i \cap \overline{S}$. 
    Take another vertex $y = (i'',j) \in \overline{S}$ in the same horizontal layer as $x$. Then Proposition \ref{prop: delta formula} implies $$\Delta_{\overline{S}}(x,y) = a_i + a_{i''} + 2 = (t+2)+(t+1) = 2t+3,$$ and hence $\Delta_{S}(x,y) = (2n+2) - (2t+3) = 2n - 2t - 1$, contradicting that $S$ is a weak $(2n-2t)$-resolving set.
    Thus, there is no weak $(2n-2t)$-resolving set of $K_n \times K_n$ with size at most $n^2-tn-1$. Therefore, $\wdim_{2n-2t}(K_n \times K_n) \geq n^2-tn$ and the desired equality follows.
\end{proof}

\subsection{Odd case for $k$}

We begin by presenting an observation similar to Lemma \ref{lem: <= 2t+1}.

\begin{lemma}
\label{lem: <= 2t+2}
    Let $n \geq 4$ and $t \in [n-1]$ be integers. Let $S$ be a weak $(2n-2t-1)$-resolving set of $K_n \times K_n$. For every $i,i' \in [n]$,
    \begin{equation}
        |(L_i \cup L_{i'}) \cap \overline{S}| \leq
        \begin{cases}
            2t+1, &\text{if $(i,j) \in \overline{S}$ and $(i',j) \in \overline{S}$ for some $j$;}\\
            2t+2, &\text{if otherwise.}
        \end{cases}
    \end{equation}
\end{lemma}

\begin{proof}
    Fix $i,i' \in [n]$. Suppose there is some $j$ such that $x = (i,j) \in \overline{S}$ and $y =(i',j) \in \overline{S}$. If $|(L_i \cup L_{i'}) \cap \overline{S}| \geq 2t+2$, then Proposition \ref{prop: delta formula} implies $\Delta_{\overline{S}}(x,y) \geq 2t + 2 + 2 = 2t+4$, and thus $\Delta_S(x,y) \leq (2n+2) - (2t+4) = 2n - 2t - 2$, contradicting that $S$ is a weak $(2n-2t-1)$-resolving set. 
    
    Now, let us assume otherwise that such $j$ does not exist. Suppose that $|(L_i \cup L_{i'}) \cap \overline{S}| \geq 2t+3$. Take any $x = (i,j') \in \overline{S}$ and $y = (i',j') \notin \overline{S}$. Hence, Proposition \ref{prop: delta formula} implies $\Delta_{\overline{S}}(x,y) \geq (2t+3)+1 = 2t+4$, and so $\Delta_S(x,y) \leq (2n+2) - (2t+4) = 2n -2t - 2$, which is a contradiction.
\end{proof}

\begin{theorem}
    For every integer $n \geq 6$ and $1 \leq t \leq n-3$, $$\wdim_{2n-2t-1}(K_n \times K_n) = n^2 - tn - \floor{\frac{n}{t+1}}.$$
\end{theorem}

\begin{proof}
    We first establish the upper bound. Following the construction used in Theorem \ref{thm: wdim even k}, we define
    \begin{align*}
        D_0 &= \{(i,i) : i \in [n]\}, \\ 
        D_j &= \{(i+j,i) : i \in [n-j]\} \cup \{(i,n-j+i) : i \in [j]\} \quad (1 \leq j \leq t-1), \\
        D_t &= \{(i,i-t) : i \in [n], i \equiv 0 \pmod{t+1}\}.
    \end{align*}
    We claim that the set $S = V \setminus  \bigcup_{j=0}^{t} D_j$
    is a weak $(2n-2t-1)$-resolving set of size $n^2 - tn - \floor{\frac{n}{t+1}}$. See Fig. \ref{fig: odd k} for illustration. By construction, each layer $L_i$ and $L^j$ satisfy
    \begin{equation*}
        |L_i \cap S| =
        \begin{cases}
            n-t-1, & i \equiv 0 \pmod{t+1}; \\
            n-t, & \text{otherwise},
        \end{cases}
    \end{equation*}
    and
    \begin{equation*}
        |L^j \cap S| =
        \begin{cases}
            n-t-1, & j \equiv 1 \pmod{t+1}; \\
            n-t, & \text{otherwise}.
        \end{cases}
    \end{equation*}
    Moreover, if $i \equiv i' \equiv 0 \pmod{t+1}$, then there is no $j \in [n]$ such that both $(i,j) \notin S$ and $(i',j) \notin S$. This property behaves similarly for the layers $L^j$ and $L^{j'}$ if $j \equiv j' \equiv 1 \pmod{t+1}$.

    \begin{figure}
        \centering
        \newcommand{\n}{14}
        
        \begin{tikzpicture}[scale=0.55]
            \foreach \i in {1,...,\n}
                \foreach \j in {1,...,\n}
                    \node[circle,fill=black,inner sep=1.8pt] (\i-\j) at (\i,\j) {};
            
            \foreach \p in {
                1-1,2-2,3-3,4-4,5-5,6-6,7-7,8-8,9-9,10-10,11-11,12-12,13-13,14-14,
                2-1,3-2,4-3,5-4,6-5,7-6,8-7,9-8,10-9,11-10,12-11,13-12,14-13,
                3-1,4-2,5-3,6-4,7-5,8-6,9-7,10-8,11-9,12-10,13-11,14-12,
                4-1,5-2,6-3,7-4,8-5,9-6,10-7,11-8,12-9,13-10,14-11,
                5-1,6-2,7-3,8-4,9-5,10-6,11-7,12-8,13-9,14-10,
                1-14,1-13,1-12,1-11,2-14,2-13,2-12,3-14,3-13,4-14,
                6-1,12-7}
                \node[circle,draw,fill=white,inner sep=1.8pt] at (\p) {};
        \end{tikzpicture}
        
        \caption{The weak $(2n-2t-1)$-resolving set $S$ of $K_n \times K_n$ for $(n,t) = (14,5)$}
        
        \label{fig: odd k}
    \end{figure}
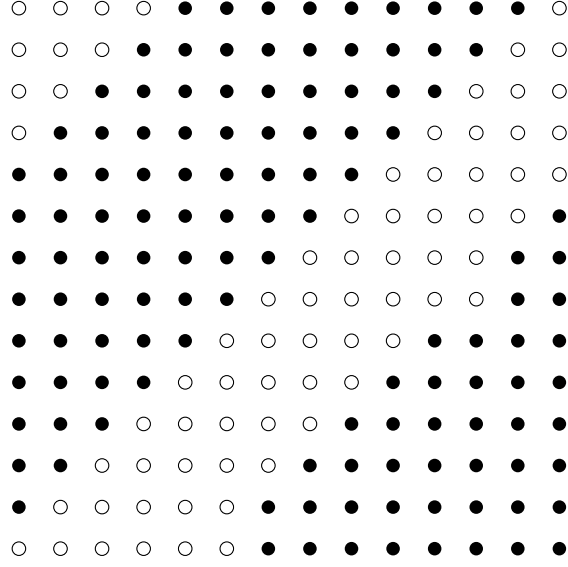

    Consider two vertices $x = (i,j)$ and $y = (i',j)$ lying in the same horizontal layer. The following argument also applies if the two vertices lie in the same vertical layer. If at least one of $x,y$ belongs to $S$, then Proposition \ref{prop: delta formula} implies $$\Delta_S(x,y) \geq 2(n-t-1) + 1 = 2n - 2t-1.$$
    If neither $x$ nor $y$ belongs to $S$, then at least one of $i,i'$ is not a multiple of $t+1$, and the same inequality holds: $$\Delta_S(x,y) \geq (n-t-1) + (n-t) + 0 = 2n-2t-1.$$
    
    Now, consider the pair $x = (i,j), y = (i',j')$ that lies in different layers. We have two cases based on the value of $t$ where $t \le n-3$.

    \medskip
    \noindent
    \textbf{Case 1:} $t \leq n-5$. Proposition \ref{prop: delta formula} implies
    $$\Delta_S(x,y) \geq 4(n-t-1) - 2 - 2(2) = 4n-4t-10 \geq 2n-2t-1.$$

    \medskip
    \noindent
    \textbf{Case 2:} $t \in \{n-4,n-3\}$. Notice that $t+1$ is the only multiple of $t+1$ in $[n]$, so we only have two layers $L_{t+1}$ and $L^1$ that contain $n-t-1$ vertices in $S$. We classify some cases depending on whether $L_{t+1}$ or $L^1$ is in $\{L_i,L_{i'}\}$ or $\{L^j,L^{j'}\}$, respectively. If both layers are involved, then $x = (t,1)$ or $y = (t,1)$ which is not in $S$. Proposition \ref{prop: delta formula} implies
    $$\Delta_S(x,y) \geq 2(n-t) + 2(n-t-1) - 1 - 2(2) = 4n - 4t - 7 \geq 2n - 2t -1.$$ 
    If exactly one of the layers is involved, then
    $$\Delta_S(x,y) \geq 3(n-t) + (n-t-1) - 2 - 2(2) = 4n - 4t - 7 \geq 2n - 2t -1.$$
    If neither layer is involved, then 
    $$\Delta_S(x,y) \geq 4(n-t) - 2 - 2(2) = 4n - 4t - 6 \geq 2n - 2t -1.$$
    
    In every case, we have $\Delta_S(x,y) \geq 2n-2t-1$. Thus, $S$ is a weak $(2n-2t-1)$-resolving set, and thus, $\wdim_{2n-2t-1}(K_n \times K_n) \le n^2 - tn - \floor{\frac{n}{t+1}}$.

    \medskip
    We now prove the lower bound. Suppose to the contrary, that there is some weak $(2n-2t-1)$-resolving set $S$ with $|S| \leq n^2-tn-\floor{\frac{n}{t+1}} - 1$, which means $|\overline{S}| \geq tn + \floor{\frac{n}{t+1}} + 1$. Then, there must exist a layer $L_i$ such that $a_i = |L_i \cap \overline{S}| \geq t+1$. Hence, Lemma \ref{lem: <= 2t+2} implies $a_{i'} \leq t+1$ for every $i' \neq i$. We consider two cases.

    \medskip
    \noindent
    \textbf{Case 1:} $a_i = t + 1$. We partition $[n]$ into two sets: $A = \{r \in [n] : a_r = t+1\}$ and $B = \{r \in [n] : a_r \leq t\}$. For distinct $i,i' \in A$, we have $a_i + a_{i'} = 2t+2$, and so Lemma \ref{lem: <= 2t+2} implies that the holes in $L_i$ and $L_{i'}$ occupy disjoint horizontal layers. Consequently, $\sum_{i \in A} a_i = |A|(t+1)$. However, since there are $n$ horizontal layers, we must have $|A| \leq \floor{\frac{n}{t+1}}$. On the other hand, $\sum_{i \in B} a_i \leq |B|t = (n-|A|)t$. Then,
    $$|\overline{S}| = \sum_{i=1}^n a_i \leq |A|(t+1) + (n-|A|)t = tn + |A| \leq tn + \floor{\frac{n}{t+1}},$$ contradicting the assumed lower bound on $|\overline{S}|$.

    \medskip
    \noindent
    \textbf{Case 2:} $a_i \geq t + 2$.
    Then, Lemma \ref{lem: <= 2t+2} gives $a_{i'} \leq t$ for every $i' \neq i$. We claim that $a_{i'} = t$ for every $i' \neq i$. If $a_{i''} \leq t-1$ for some $i'' \neq i$, then
    $$a_i = |\overline{S}| - \sum_{i' \neq i} a_{i'} \geq tn + \floor{\frac{n}{t+1}} + 1 - (t-1) - (n-2)t = t + \floor{\frac{n}{t+1}} + 2 \geq t + 3$$ since $n \geq t+2$. Since $a_i \ge t+3$, then again $a_{i'} \leq t-1$ for every $i' \neq i$, yielding
    $$a_i \geq tn + \floor{\frac{n}{t+1}} + 1 - (n-1)(t-1) = t + n + \floor{\frac{n}{t+1}} \geq 2t+3,$$ contradicting Lemma \ref{lem: <= 2t+2}. Thus, $a_{i'} = t$ for every $i' \neq i$. 
    We have two subcases based on the value of $t \geq 1$.

    \medskip
    \noindent
    \textbf{Subcase 2.1:} $t \ge 2$.
    By the symmetry of $K_n \times K_n$, the same configuration holds in the horizontal layers, that is, each horizontal layer must contain at least $t \geq 2$ holes. Take any $j \in [n]$ such that $x = (i,j) \in \overline{S}$ and $y = (i'',j) \in \overline{S}$ with $i'' \neq i$. Then
    $$\Delta_{\overline{S}}(x,y) = a_i + a_{i''} + 2 \ge (t+2)+t+2 = 2t+4,$$ and therefore $\Delta_S(x,y) \le (2n+2)-(2t+4) = 2n-2t-2$, a contradiction.

    \medskip
    \noindent
    \textbf{Subcase 2.2:} $t = 1$. Once more, Lemma \ref{lem: <= 2t+2} implies $a_i \le t+2$ and hence $a_i=t+2$. Therefore, the total number of holes is $|\overline{S}| = (t+2) + (n-1)t = tn + 2 < tn + \floor{\frac{n}{t+1}}$ for $n \ge 6$, contradicting the assumed lower bound on $|\overline{S}|$.

    \medskip
    As a consequence of all the arguments above, we obtain that  $|S| \leq n^2-tn-\floor{\frac{n}{t+1}} - 1$ is not possible. Therefore, it must happen  $\wdim_{2n-2t-1}(K_n \times K_n) \geq n^2-tn-\floor{\frac{n}{t+1}}$, which completes the proof of the desired equality.
\end{proof}

\section{Concluding remarks}

We have computed in this work the weak $k$-metric dimension of the direct product graph $K_n\times K_n$ in all but one situation. In this sense, the following open questions might be of interest to continue the investigation.
\begin{itemize}
    \item Since the case $k=4$ is the only one in which an exact formula was not deduced, it is clearly of interest to complete this case, and we wonder on whether the exact value is precisely that shown as an upper bound in Theorem \ref{th:k-4}.
    \item Natural generalizations of our results can be deduced when the direct product graphs $K_n\times K_m$ are considered for any two integers $n,m\ge 3$, as well as, the direct product of more than two complete graphs.
    \item The study of other direct product graphs $G\times K_n$ when $G$ is a cycle, a path, or a tree would be also of interest.
    \item The weak $k$-metric dimension of other related product graphs like the strong one could be also an interesting research line.
\end{itemize}

\section*{Acknowledgment}

The authors have been partially supported by ``Ministerio de Ciencia, Innovaci\'on y Universidades'' through the grant PID2023-146643NB-I00.


\begin{thebibliography}{10}

\bibitem{Bailey-2023} 
R.~F.~Bailey, P.~Spiga,
Metric dimension of dual polar graphs, 
Arch.\ Math.\ 120 (2023) 467--478.

\bibitem{Blumenthal} 
L.~M.~Blumenthal, 
Theory and Applications of Distance Geometry. 
Oxford University Press (1953).

\bibitem{Chartrand}
G.~Chartrand, L.~Eroh, M.~A.~Johnson, O.~R.~Oellermann,
Resolvability in graphs and the metric dimension of a graph,
Discrete Appl.\ Math.\ 105(1-3) (2000) 99--113.

\bibitem{Dankelmann-2023} 
P.~Dankelmann, J.~Morgan, E.~Rivett-Carnac, 
Metric dimension and diameter in bipartite graphs, 
Discuss.\ Math.\ Graph Theory 43 (2023) 487--498.

\bibitem{Estrada-Moreno2013}
A.~Estrada-Moreno, J.~A.~Rodr\'{\i}guez-Vel\'{a}zquez, I.~G.~Yero, 
The $k$-metric dimension of a graph, 
Appl.\ Math.\ Inf.\ Sci.\ 9 (2015) 2829--2840.

\bibitem{Elena} 
E.~Fernandez, S.~Klav\v{z}ar, D.~Kuziak, M.~Mu\~noz-M\'arquez, I.~G.~Yero,
On the weak $k$-metric dimension of Hamming graphs,
Discrete Opt.\ 60 (2026) article 100945.

\bibitem{foster-2024} 
B.~Foster-Greenwood, Ch.~Uhl, 
Metric dimension of a direct product of three complete graphs, Electron.\ J.\ Combin.\ 31 (2024) article 2.13.

\bibitem{Hakanen}
A.~Hakanen, V.~Junnila, T.~Laihonen, I.~G.~Yero, 
On vertices contained in all or in no metric basis, Discrete Appl.\ Math.\ 319 (2022) 407--423.

\bibitem{Harary1976}
F.~Harary, R.~A.~Melter, 
On the metric dimension of a graph, 
Ars Combin.\ 2 (1976) 191--195.

\bibitem{kuziak17} D. Kuziak, I. Peterin, I. G. Yero, Resolvability and strong resolvability in the direct product of graphs, Results Math.\ 71 (1) (2017) 509--526.

\bibitem{Kuziak} D.~Kuziak, I.~G.~Yero, 
Metric dimension related parameters in graphs: A survey on combinatorial, computational and applied results, 
arXiv:2107.04877 [math.CO].

\bibitem{Peterin} 
I.~Peterin, J.~Sedlar, R.~\v{S}krekovski, I.~G.~Yero, 
Resolving vertices of graphs with differences, 
Comput.\ Appl.\ Math.\ 43 (2024) article 275.

\bibitem{Prabhu} 
S.~Prabhu, T.~J.~Janany and S.~Klav\v zar, 
Metric dimensions of generalized Sierpi\'nski graphs over squares, Appl. Math. Comput. 505 (2025), Paper No. 129528.

\bibitem{Slater1975}
P.~J.~Slater, 
Leaves of trees, 
Cong.\ Numer.\ 14 (1975) 549--559.

\bibitem{Tillquist} 
R.~C.~Tillquist, R.~M.~Frongillo, M.~E.~Lladser, 
Getting the lay of the land in discrete space: A survey of metric dimension and its applications, 
SIAM Rev.\ 65 (2023) 919--962.

\end{thebibliography}
\end{document}